\documentclass{amsart}
\usepackage{amsfonts,amssymb,amsmath,amsthm}
\usepackage{url}
\usepackage{enumerate}
\input{mathrsfs.sty}

\newcommand{\A}{\mathscr A}

\newcommand{\perpb}{\perp_B}
\newcommand{\perpm}{\perp^s_B}

\newcommand{\perpp}{\perp^s_{B^\varepsilon}}
\newcommand{\perps}{\perp^{\varepsilon}_B}
\newcommand{\perpr}{\perp^\varepsilon}

\newcommand{\la}{\langle}
\newcommand{\ra}{\rangle}
\newcommand{\K}{\mathbb{K}}
\newcommand{\B}{\mathbb{B}}

\newtheorem{theorem}{Theorem}[section]

\newtheorem{proposition}[theorem]{Proposition}
\newtheorem{corollary}[theorem]{Corollary}
\theoremstyle{definition}
\newtheorem{definition}[theorem]{Definition}
\newtheorem{example}[theorem]{Example}

\theoremstyle{remark}
\newtheorem{remark}[theorem]{Remark}
\numberwithin{equation}{section}

\begin{document}
\title[Characterizations of operator Birkhoff–-James orthogonality]{Characterizations of operator Birkhoff–-James orthogonality}

\author[M.S. Moslehian and A. Zamani]{Mohammad Sal Moslehian$^1$ and Ali Zamani$^{1,2}$}

\address{$^1$Department of Pure Mathematics, Ferdowsi University of
Mashhad, P.O. Box 1159, Mashhad 91775, Iran}
\email{moslehian@um.ac.ir}
\address{$^2$Department of Mathematics, Farhangian University, Iran}
\email{zamani.ali85@yahoo.com}



\begin{abstract}
In this paper, we obtain some characterizations of the (strong) Birkhoff--James orthogonality for elements of Hilbert $C^*$-modules
and certain elements of $\B(\mathscr{H})$.
Moreover, we obtain a kind of Pythagorean relation for bounded linear operators.
In addition, for $T\in \B(\mathscr{H})$ we prove that if the norm attaining
set $\mathbb{M}_T$ is a unit sphere of some finite dimensional
subspace $\mathscr{H}_0$ of $\mathscr{H}$ and $\|T\|_{{{\mathscr{H}}_0}^\perp} < \|T\|$, then for every $S\in\B(\mathscr{H})$, $T$ is the strong Birkhoff--James orthogonal to $S$ if and only if there exists a unit vector $\xi\in {\mathscr{H}}_0$ such that $\|T\|\xi = |T|\xi$ and $S^*T\xi = 0$.
Finally, we introduce a new type of approximate orthogonality and investigate this notion in the setting of inner product $C^*$-modules.
\end{abstract}
\maketitle
%%%%%%%%%%%%%%%%%%%%%%%%%%%%%%%%%%%%%%%%%%%%%%%%%%%%%%%%%%%%%%%%%%%%%%%%%%%%%%%%%%%%%%%%%%%%%%%%%%%%%%%%%%%
\section{Introduction and preliminaries}
%%%%%%%%%%%%%%%%%%%%%%%%%%%%%%%%%%%%%%%%%%%%%%%%%%%
Let $\B(\mathscr{H}, \mathscr{K})$ denote the linear space of all bounded linear operators
between Hilbert spaces $(\mathscr{H}, [., .])$ and $(\mathscr{K}, [., .])$. By $I$ we denote the identity
operator. When $\mathscr{H} = \mathscr{K}$, we write $\B(\mathscr{H})$ for $\B(\mathscr{H}, \mathscr{K})$. By $\K(\mathscr{H})$ we
denote the algebra of all compact operators on $\mathscr{H}$, and by $\mathcal{C}_1(\mathscr{H})$ the algebra
of all trace--class operators on $\mathscr{H}$.
Let $\mathbb{S}_{\mathscr{H}} = \{\xi\in \mathscr{H}: \, \|\xi\| = 1\}$ be the unit sphere of $\mathscr{H}$.
For $T\in \B(\mathscr{H})$, let $\mathbb{M}_T$ denote the set of all vectors in $\mathbb{S}_{\mathscr{H}}$ at which $T$ attains norm, i.e.,
$\mathbb{M}_T = \{\xi\in \mathbb{S}_{\mathscr{H}}: \, \|T\xi\| = \|T\|\}.$
For $T\in \B(\mathscr{H}, \mathscr{K})$ the symbol $m(T) : = \inf\{\|T\xi\|: \, \xi\in \mathbb{S}_{\mathscr{H}}\}$ denotes the minimum modulus of $T$.

Inner product $C^{*}$-modules generalize inner product spaces by allowing inner products to take values
in an arbitrary $C^{*}$-algebra instead of the $C^{*}$-algebra of complex numbers.

In an inner product $C^*$-module $(V, \langle\cdot,\cdot\rangle)$ over a $C^{*}$-algebra $\A$ the following Cauchy--Schwarz inequality holds (see also \cite{A.B.F.M}):
$${\langle x, y\rangle}^*\langle x, y\rangle \leq \|\langle x, x\rangle\|\langle y, y\rangle \qquad(x, y\in V).$$
Consequently, $\|x\|=\|\langle x, x\rangle\|^{\frac{1}{2}}$ defines a norm on $V$. If $V$ is complete with respect to this norm, then it is called a Hilbert $\A$-module, or a Hilbert $C^*$-module over $\A$.
Any $C^*$-algebra $\A$ can be regarded as a Hilbert $C^*$-module over itself via $\langle a, b\rangle :=a^* b$. For every $x\in V$ the positive square root of $\langle x, x\rangle$ is denoted by $|x|$.
In the case of a $C^*$-algebra, we get the usual notation $|a| = (a^*a)^{\frac{1}{2}}.$
By $\mathcal{S}(\A)$ we denote the set of all states of $\A$, that is, the set of all
positive linear functionals of $\A$ whose norm is equal to one.

Furthermore, if $\varphi \in \mathcal{S}(\A)$, then $(x, y)\mapsto\varphi(\langle x, y\rangle)$ gives rise to a usual semi-inner product on $V$, so we have the following useful Cauchy–-Schwarz inequality:
$$|\varphi(\langle x, y\rangle)|^2\leq\varphi(\langle x, x\rangle)\varphi(\langle y, y\rangle)\qquad(x, y\in V).$$

We refer the reader to \cite{dix, lan, mor} for more information on the basic theory of $C^{*}$-algebras and Hilbert $C^*$-modules.
%%%%%%%%%%%%%%%%%%%%%%%%%%%%%%%%%%%%%%%%%%%%%%%%%%%%%%%%%%%%%%%%%%%%%%%%%%%%%%%%%%%%%%%

A concept of orthogonality in a Hilbert $C^{*}$-module can be defined with respect to the $C^{*}$-
valued inner product in a natural way, that is, two elements $x$ and $y$ of a Hilbert $C^{*}$-module
$V$ over a $C^{*}$-algebra $\A$ are called orthogonal, in short $x \perp y$, if $\langle x, y\rangle = 0$.

In a normed linear space there are several notions of orthogonality, all of which are
generalizations of orthogonality in a Hilbert space.
One of the most important is concept of the Birkhoff--James orthogonality: if $x, y$ are elements of a complex
normed linear space $(X, \|\cdot\|)$, then $x$ is orthogonal to $y$ in the Birkhoff--James sense \cite{B, J}, in short $x \perpb y$, if
$$\|x + \lambda y\| \geq \|x\| \qquad (\lambda\in\mathbb{C}).$$
The central role of the Birkhoff--James orthogonality in approximation theory, typified by the fact
that $T \in\B(\mathscr{H})$ is a best approximation of $S\in\B(\mathscr{H})$ from a linear subspace $M$ of $\B(\mathscr{H})$ if and only if
$T$ is a Birkhoff--James orthogonal projection of $S$ on to $M$.
By the Hahn–-Banach theorem, if $x, y$ are two elements of a normed linear space $X$,
then $x \perpb y$ if and only if there is a norm one linear functional $f$ of $X$ such that
$f(x) = \|x\|$ and $f(y) = 0$. If we have additional structures on a normed
linear space $X$, then we obtain other characterizations of the Birkhoff–-James
orthogonality see \cite{A.R.3, B.G, G.S.P, S.P.H, Z.M.F} and the references therein.

In Section 2, we present some characterizations of the Birkhoff--James orthogonality for elements of a Hilbert $\K(\mathscr{H})$-module and elements of $\B(\mathscr{H})$. Next, we will give some applications.
In particular, for $T, S\in\B(\mathscr{H})$ with $m(S)> 0$,
we prove that there exists a unique $\gamma\in\Bbb C$ such that
$$\Big\|(T + \gamma S) + \lambda S \Big\|^2 \geq \Big\|T + \gamma S \Big\|^2 + |\lambda|^2m^2(S) \qquad (\lambda\in\Bbb C).$$
As a natural generalization of the notion of Birkhoff--James orthogonality, the concept of strong Birkhoff--James orthogonality, which involves modular structure of a Hilbert $C^{*}$-module was introduced in \cite{A.R.2}. When $x$ and $y$ are elements of a Hilbert $\A$-module $V$, $x$ is orthogonal to $y$ in the strong Birkhoff--James sense, in short $x \perpm y$, if
$$\|x + ya\| \geq \|x\| \qquad (a\in \A),$$
i.e. if the distance from $x$ to $\overline{y\A}$, the $\A$-submodule of $V$ generated by $y$, is exactly $\|x\|$.
This orthogonality is ``between'' $\perp$ and $ \perpb $, i.e.,
$$x \perp y \, \Longrightarrow \, x\perpm y \, \Longrightarrow \, x\perpb y,\qquad (x, y \in V),$$
while the converses do not hold in general (see \cite{A.R.2}). It was shown in \cite{A.R.2} that the following relation
between the strong and the classical Birkhoff–-James orthogonality is valid:
\begin{align*}
x\perpm y \Leftrightarrow  x\perpb y\la y, x\ra \qquad (x, y \in V).
\end{align*}
In particular, by \cite[Proposition 3.1]{A.R.3}, if $\la x, y\ra \geq 0$, then
\begin{align}\label{id.1.102}
x\perpm y \Leftrightarrow  x\perpb y \qquad (x, y \in V).
\end{align}
If $V$ is a full Hilbert $\A$-module, then the only
case where the orthogonalities $\perpm$ and $\perpb$ coincide is when $\A$ is isomorphic to $\mathbb{C}$
(see \cite[Theorem 3.5]{A.R.3}), while orthogonalities $\perpm$ and $\perp$ coincide only when $\A$ or
$\K(V)$ is isomorphic to $\mathbb{C}$ (see \cite[Theorems 4.7, 4.8]{A.R.3}).
Further, by \cite[Lemma 4.2]{A.R.3}, we have
\begin{align*}
x\perpb \big(\|x\|^2\,y - y\la x, x\ra\big) \qquad (x, y \in V),
\end{align*}
and
\begin{align}\label{id.1.104}
x\perpm \big(\|x\|^2x-x\la x,x\ra\big) \qquad (x\in V).
\end{align}
In Section 2, we obtain a characterization of the strong Birkhoff--James orthogonality for elements of a $C^*$-algebra.
We also present some characterizations of the strong Birkhoff--James orthogonality for certain elements of $\B(\mathscr{H}).$
In particular, for $T\in \B(\mathscr{H})$ we prove that if $\mathbb{S}_{{\mathscr{H}}_0} = \mathbb{M}_T$, where ${\mathscr{H}}_0$ is a finite dimensional
subspace of $\mathscr{H}$ and $\|T\|_{{{\mathscr{H}}_0}^\perp} < \|T\|$, then for every $S\in\B(\mathscr{H})$, $T \perpm S$ if and only if there exists a unit vector $\xi\in {\mathscr{H}}_0$ such that $\|T\|\xi = |T|\xi$ and $S^*T\xi = 0$.
%%%%%%%%%%%%%%%%%%%%

For given $\varepsilon\geq0$, elements $x, y$ in an inner product $\A$-module $V$ are said to be approximately orthogonal or $\varepsilon$-orthogonal, in short $x \perpr y$, if $\|\langle x, y\rangle\|\leq \varepsilon \|x\|\|y\|$. For $\varepsilon\geq1$, it is clear that every pair of vectors are $\varepsilon$-orthogonal, so the interesting case is when $\varepsilon\in[0, 1)$.

In an arbitrary normed space $X$, Chmieli\'{n}ski \cite{Ch, Ch.2} introduced the approximate Birkhoff--James orthogonality $x\perps y$ by
$$\|x + \lambda y\|^2\geq\|x\|^2-2\varepsilon\,|\lambda|\,\|x\|\, \|y\| \qquad (\lambda\in\mathbb{C}).$$
Inspired by the above the approximate Birkhoff--James orthogonality, we propose a new type of approximate orthogonality in inner product $C^{*}$-modules.
\begin{definition}\label{df.011}
For given $\varepsilon\in[0 , 1)$ elements $x, y$ of an inner product $\A$-module $V$ are said to be approximate strongly Birkhoff–-James orthogonal, denoted by $x\perpp y$, if
$$\|x + ya\|^2\geq\|x\|^2-2\varepsilon\,\|a\|\,\|x\|\,\|y\|\qquad (a\in\A).$$
\end{definition}
In Section 3, we investigate this notion of approximate orthogonality in inner product $C^*$-modules.
In particular, we show that
$$x \perpr y \, \Longrightarrow \, x\perpp y \, \Longrightarrow \, x\perps y,\qquad (x, y \in V),$$
while the converses do not hold in general.

As a result, we show that if $T\,: V \longrightarrow W$ is a linear mapping between inner
product $\A$-modules such that $x \perpb y \Longrightarrow Tx\perpp Ty$ for all $x, y\in V$, then
$$(1 - 16\varepsilon)\,\|T\|\,\|x\| \leq \|Tx\| \leq \|T\|\,\|x\| \qquad (x\in V).$$
Some other related topics can be found in \cite{G, I.T, W, Z.M}.

%%%%%%%%%%%%%%%%%%%%%%%%%%%%%%%%%%%
%%%%%%%%%%%%%%%%%%%%%%%%%%%%%%%%%%%
\section{Operator (strong) Birkhoff–-James orthogonality}
The characterization of the (strong) Birkhoff–-James orthogonality for elements of a
Hilbert $C^*$-module by means of the states of the underlying $C^*$-algebra are known.
For elements $x, y$ of a Hilbert
$\A$-module $V$ the following results were obtained in \cite{A.R.2, B.G}:
\begin{align}\label{id.1.302}
x\perpb y \Leftrightarrow \big(\exists \,\varphi \in \mathcal{S}(\A): \varphi(\la x, x\ra)=\|x\|^2 \,\,\mbox{and}\,\, \varphi(\la x, y\ra)=0\big)
\end{align}
and
\begin{align}\label{id.1.303}
x\perpm y \Leftrightarrow \big(\exists \,\varphi \in \mathcal{S}(\A): \varphi(\la x, x\ra)=\|x\|^2 \,\,\mbox{and}\,\, \varphi(\la x,y\ra a)=0 \,\,\forall \,a\in \A\big).
\end{align}
%%%%%%%%%%%%%%%%%%%%%%%%%%%%%%%%%%%%%%%%%%%%%%%%%%%%%%%%%%%%%
In the following result we establish a characterization of the Birkhoff--James orthogonality for elements of a Hilbert $\K(\mathscr{H})$-module.
%%%%%%%%%%%%%%%%%%%%%%%%%%%%%%%%%%%%%%
\begin{theorem}\label{cr.024}
Let $V$ be a Hilbert $\K(\mathscr{H})$-module and $x , y\in V$. Then the following statements are equivalent\textup{:}
\begin{itemize}
\item[(i)] $x\perpb y$.
\item[(ii)] There exists a positive operator $P\in\mathcal{C}_1(\mathscr{H})$ of trace one such that
$$\|x + \lambda y\|^2 \geq \|x\|^2 + |\lambda|^2{\rm tr}(P|y|^2) \qquad (\lambda\in\Bbb C).$$
\end{itemize}
\end{theorem}
\begin{proof}
Let $x\perpb y$. By (\ref{id.1.302}), there exists a state $\varphi$ over $\K(\mathscr{H})$ such that $\varphi(\la x, x\ra)=\|x\|^2$ and $\varphi(\la x, y\ra)=0.$ For every $\lambda\in\Bbb C$, we therefore have
\begin{align*}
\|x + \lambda y\|^2 & \geq \varphi\Big(\la x + \lambda y, x + \lambda y\ra\Big)
\\& = \varphi(\la x, x\ra) + \lambda\varphi(\la x, y\ra) + \overline{\lambda\varphi(\la x, y\ra)} + |\lambda|^2\varphi(\la y, y\ra)
\\& = \|x\|^2 + |\lambda|^2\varphi(|y|^2).
\end{align*}
Thus
$$\|x + \lambda y\|^2 \geq \|x\|^2 + |\lambda|^2\varphi(|y|^2) \qquad (\lambda\in\Bbb C).$$
Now, by \cite[Theorem 4.2.1]{mor}, there exists a positive operator $P\in\mathcal{C}_1(\mathscr{H})$ of trace one such that $\varphi(T) = {\rm tr}(PT)$, $T\in \K(\mathscr{H})$. Thus we have $$\|x + \lambda y\|^2 \geq \|x\|^2 + |\lambda|^2\varphi(|y|^2) = \|x\|^2 + |\lambda|^2{\rm tr}(P|y|^2) \qquad(\lambda\in\Bbb C).$$
Conversely, if (ii) holds then, since $|\lambda|^2{\rm tr}(P|y|^2)\geq 0$ for all $\lambda\in\Bbb C$, we get
$$\|x + \lambda y\| \geq \sqrt{\|x\|^2 + |\lambda|^2{\rm tr}(P|y|^2)} \geq \|x\| \qquad (\lambda\in\Bbb C).$$
Hence $x\perpb y$.
\end{proof}
%%%%%%%%%%%%%%%%%%%%%%%%%%%%%%%%%%%%%%%%%%%%%%%%%%%%%%%%%%%%%%%%%%%%%%%%%%%%%%%%%%%%%%%%%%%%%%%%%%%%%
\begin{remark}
Let $V$ be a Hilbert $\K(\mathscr{H})$-module and $x , y\in V$. Using the same argument as in the proof of Theorem \ref{cr.024} and (\ref{id.1.303}) we obtain
$x\perpm y$ if and only if there exists a positive operator $P\in\mathcal{C}_1(\mathscr{H})$ of trace one such that
  $$\|x + ya\|^2 \geq \|x\|^2 + {\rm tr}(P|ya|^2) \qquad (a\in \A).$$
\end{remark}
%%%%%%%%%%%%%%%%%%%%%%%%%%%%%%%%%%%%%%%%%%%%%%%%%%%%%%
In the following result we establish a characterization of the strong Birkhoff--James orthogonality for elements of a $C^*$-algebra.
\begin{theorem}\label{cr.0212}
Let $\A$ be a $C^*$-algebra, and $a, b\in \A$. Then the following statements are equivalent\textup{:}
\begin{itemize}
\item[(i)] $a\perpm b$.
\item[(ii)] There exist a Hilbert space $\mathscr{H}$, a representation $\pi:\A\to \B(\mathscr{H})$ and a unit vector $\xi\in \mathscr{H}$ such that
$$\|a + bc\|^2 \geq \|a\|^2 + \|\pi(bc)\xi\|^2 \qquad (c\in \A).$$
\end{itemize}
\end{theorem}
\begin{proof}
Suppose that $a\perpm b$. By (\ref{id.1.303}) applied to $V = \A$ and using the same argument as in the proof of Theorem \ref{cr.024}, there exists a state $\varphi$ of $\A$ such that $\|a + bd\|^2 \geq \|a\|^2 + \varphi(|bd|^2)$ for all $d\in \A$. Now, by \cite[Proposition 2.4.4]{dix} there exist a Hilbert space $\mathscr{H}$, a representation $\pi:\A\to \B(\mathscr{H})$ and a unit vector $\xi\in \mathscr{H}$ such that for any $c\in\A$ we have $\varphi(c)=[\pi(c)\xi, \xi].$ Hence
\begin{align*}
\|a + bc\|^2 &\geq \|a\|^2 + \varphi(|bc|^2) = \|a\|^2 + \big[\pi(|bc|^2)\xi, \xi\big]
\\& = \|a\|^2 + \big[\pi(bc)\xi, \pi(bc)\xi\big] = \|a\|^2 + \|\pi(bc)\xi\|^2,
\end{align*}
for all $c\in\A.$

The converse is obvious.
\end{proof}
%%%%%%%%%%%%%%%%%%%%%%%%%%%%%%%%%%%%%%%%%%%%%
\begin{corollary}\label{cr.0213}
Let $\A$ be a unital $C^*$-algebra with the unit $e$. For every self-adjoint
noninvertible $a\in \A$, there exist a Hilbert space $\mathscr{H}$, a representation $\pi:\A\to \B(\mathscr{H})$ and a unit vector $\xi\in \mathscr{H}$ such that
$$\|e + ab\|^2 \geq 1 + \|\pi(ab)\xi\|^2 \qquad (b\in \A).$$
\end{corollary}
\begin{proof}
Since $a$ is noninvertible, $a^2$ is noninvertible as well. Therefore there is a state $\varphi$ of $\A$ such that $\varphi(a^2) = 0$. We have $\varphi(ee^\ast) = \|e\|^2 = 1$ and
$$|\varphi(eab)|\leq \sqrt{\varphi(eaa^\ast e^\ast)\varphi(b^\ast b)} = \sqrt{\varphi(a^2)\varphi(b^\ast b)} = 0 \qquad(b \in \A).$$
Thus by (\ref{id.1.303}) we get $e\perpm a$.
Hence, by Theorem \ref{cr.0212}, there exist a Hilbert space $\mathscr{H}$, a representation $\pi:\A\to \B(\mathscr{H})$ and a unit vector $\xi\in \mathscr{H}$ such that
$\|e + ab\|^2 \geq 1 + \|\pi(ab)\xi\|^2$ for all $b\in \A$.
\end{proof}
%%%%%%%%%%%%%%%%%%%%%%%%%%%%%%%%%%%%%%%%%%%%%%%%%%%%%%
Now, we are going to obtain some characterizations of the (strong) Birkhoff--James orthogonality in the Hilbert $C^*$-module $\B(\mathscr{H})$.
Let $T, S\in\B(\mathscr{H})$. It was proved in \cite[Theorem 1.1 and Remark 3.1]{B.S} and \cite[Proposition 2.8]{A.R.2}, that
$T\perpb S$ (resp. $T\perpm S$) if and only if there is a sequence of unit vectors $({\xi}_n)\subset \mathscr{H}$ such that
\begin{align}\label{id.1.602}
\lim_{n\rightarrow\infty} \|T{\xi}_n\| = \|T\| \,\,\mbox{and}\,\, \lim_{n\rightarrow\infty}[T{\xi}_n, S{\xi}_n] = 0 \,(\mbox{resp.}\, \lim_{n\rightarrow\infty}S^*T{\xi}_n = 0).
\end{align}
When $\mathscr{H}$ is finite dimensional, it holds that $T\perpb S$ (resp. $T\perpm S$) if and only if
there is a unit vector $\xi\in \mathscr{H}$ such that
\begin{align}\label{id.1.603}
\|T\xi\| = \|T\| \,\,\mbox{and}\,\, [T\xi, S\xi] = 0 \,(\,\mbox{resp.}\, S^*T\xi = 0).
\end{align}
%%%%%%%%%%%%%%%%%%%%%%%%%%%%%%%%%%%%%%%%
The following results are immediate consequences of the above characterizations.
%%%%%%%%%%%%%%%%%%%%%%%%%%%%%%%%%%%%%%%%%%%%
\begin{corollary}
Let $T\in\B(\mathscr{H})$ be an isometry and $S\in\B(\mathscr{H})$ be an invertible positive operator. Then
$T\not\perpb TS$.
\end{corollary}
%%%%%%%%%%%%%%%%%%%%%%%%%%%%%%%%%
\begin{corollary}
Let $S\in\B(\mathscr{H})$. Then the following statements are equivalent\textup{:}
\begin{itemize}
\item[(i)] $S$ is non-invertible.
\item[(ii)] $T \perpb S$ for every unitary operator $T\in\B(\mathscr{H})$.
\end{itemize}
\end{corollary}
\begin{proof}
By \cite[Proposition 3.3]{Di}, $S\in\B(\mathscr{H})$ is not invertible if and only if
$$0\in \left\{\lambda\in\Bbb C: \, \exists \,({\xi}_n)\subset \mathscr{H}, \|{\xi}_n\| = 1, \lim_{n\rightarrow\infty} [T^*S{\xi}_n, {\xi}_n]= \lambda\right\},$$
for every unitary operator $T$.
Hence, by using (\ref{id.1.602}), the statements are equivalent.
\end{proof}
%%%%%%%%%%%%%%%%%%%%%%%%%%%%%%%%%
\begin{corollary}\label{rem.01}
Let $T, S\in\B(\mathscr{H})$. Then the following statements hold\textup{:}
\begin{itemize}
\item[(i)] If $\dim \mathscr{H}<\infty,$ then $T\perpb S$ if and only if there
is a unit vector $\xi\in \mathscr{H}$ such that $\|T\|\xi = |T|\xi$ and $[T\xi, S\xi] = 0.$
\item[(ii)] If $\dim \mathscr{H} = \infty,$ then $T\perpb S$ if and only if there
is a sequence of unit vectors $({\xi}_n)\subset \mathscr{H}$ such that $\lim_{n\rightarrow\infty} \big(\|T\|{\xi}_n - |T|{\xi}_n\big) = 0$
and $\lim_{n\rightarrow\infty}[T{\xi}_n, S{\xi}_n]= 0.$
\item[(iii)] If $\dim \mathscr{H}<\infty,$ then $T\perpm S$ if and only if there
is a unit vector $\xi\in \mathscr{H}$ such that $\|T\|\xi = |T|\xi$ and $S^*T\xi = 0.$
\item[(iv)] If $\dim \mathscr{H} = \infty,$ then $T\perpm S$ if and only if there
is a sequence of unit vectors $({\xi}_n)\subset \mathscr{H}$ such that $\lim_{n\rightarrow\infty} \big(\|T\|{\xi}_n - |T|{\xi}_n\big) = 0$
and $\lim_{n\rightarrow\infty}S^*T{\xi}_n = 0.$
\end{itemize}
\end{corollary}
\begin{proof}
(i) Let $T\perpb S$. Take the same vector $\xi$ as in (\ref{id.1.603}). So, we have
$${\|T\xi\|}^2 = [T\xi, T\xi] = [|T|^2\xi, \xi] \leq {\||T|\|}^2 {\|\xi\|}^2 \leq {\|T\|}^2 {\|\xi\|}^2 = {\|T\xi\|}^2$$
This forces $|T|^2\xi = \|T\|^2\xi$ and thus $|T|\xi = \|T\|\xi$, as asserted.

The converse is trivial.

Using (\ref{id.1.602}) and (\ref{id.1.603}), we can similarly prove the statements (ii)-(iv).
\end{proof}
%%%%%%%%%%%%%%%%%%%%%%%%%%%%%%%%%%%%%%%
\begin{theorem}\label{Th.321}
Let $S\in\B(\mathscr{H})$. Let $\mathscr{H}_0 \neq \{0\}$ be a closed subspace of $\mathscr{H}$ and $P$ be the orthogonal projection onto $\mathscr{H}_0$. Then the following statements hold\textup{:}
\begin{itemize}
\item[(i)] If $\dim \mathscr{H}<\infty,$ then $P\perpb S$ if and only if there
is a unit vector $\xi\in \mathscr{H}_0$ such that $[S\xi, \xi] = 0.$
\item[(ii)] If $\dim \mathscr{H} = \infty,$ then $P\perpb S$ if and only if there
is a sequence of unit vectors $({\xi}_n)\subset \mathscr{H}_0$ such that $\lim_{n\rightarrow\infty}[S{\xi}_n, {\xi}_n] = 0.$
\end{itemize}
\end{theorem}
\begin{proof}
(i) Let $P\perpb S$. By (\ref{id.1.603}), there is a unit vector $\zeta\in \mathscr{H}$ such that $\|P\zeta\| = \|P\| = 1$ and $[P\zeta, S\zeta] = 0.$
We have $\zeta = \xi + \eta$ where $\xi \in \mathscr{H}_0$ and $\eta \in {{\mathscr{H}}_0}^\perp$.
Since $\|\xi\| = \|P(\xi + \eta)\| = \|P\zeta\| = 1$ and $ \|\xi\|^2 + \|\eta\|^2 = 1$, so we get $\eta = 0$. Hence
$[S\xi, \xi] = [S(\xi + \eta), \xi] = [S(\xi + \eta), P(\xi + \eta)] = \overline{[P\zeta, S\zeta]}= 0.$

The converse is trivial.

(ii) Let $P\perpb S$. Take the vector sequence
$({\zeta}_n)$ of $\mathscr{H}$ as in
(\ref{id.1.602}).
We have ${\zeta}_n = {\mu}_n + {\eta}_n$ where ${\mu}_n \in \mathscr{H}_0$ and ${\eta}_n \in {{\mathscr{H}}_0}^\perp$.
Since $\lim_{n\rightarrow\infty} \|{\mu}_n\| = \lim_{n\rightarrow\infty} \|P({\mu}_n + {\eta}_n)\| = \lim_{n\rightarrow\infty} \|P{\zeta}_n\| = 1$ and $ \|{\mu}_n\|^2 + \|{\eta}_n\|^2 = 1$, so we get $\lim_{n\rightarrow\infty} \|{\eta}_n\| = 0$.
We may assume that $\|{\mu}_n\| \geq \frac{1}{2}$ for every $n \in \mathbb{N}$. Let us put ${\xi}_n = \frac{{\mu}_n}{\|{\mu}_n\|}$. We have
\begin{align*}
\Big|[S{\xi}_n, {\xi}_n]\Big| &= \frac{1}{\|{\mu}_n\|^2}\Big|[S{\mu}_n, {\mu}_n]\Big|
\\& = \frac{1}{\|{\mu}_n\|^2}\Big|[S{\zeta}_n, P{\zeta}_n] + [S{\mu}_n, {\mu}_n] - [S{\zeta}_n, P{\zeta}_n]\Big|
\\& \leq \frac{1}{\|{\mu}_n\|^2}\Big|[S{\zeta}_n, P{\zeta}_n]\Big| + \frac{1}{\|{\mu}_n\|^2} \Big|[S{\mu}_n, {\mu}_n] - [S({\mu}_n + {\eta}_n), {\mu}_n]\Big|
\\& \leq \frac{1}{\|{\mu}_n\|^2}\Big|[S{\zeta}_n, P{\zeta}_n]\Big| + \frac{1}{\|{\mu}_n\|^2}\Big|[S{\eta}_n, {\mu}_n]\Big|
\\& \leq \frac{1}{\|{\mu}_n\|^2}\Big|[S{\zeta}_n, P{\zeta}_n]\Big| + \frac{1}{\|{\mu}_n\|} \|S\|\|{\eta}_n\|
\\& \leq 4\Big|[S{\zeta}_n, P{\zeta}_n]\Big| + 2\|S\|\|{\eta}_n\|,
\end{align*}
whence $$\Big|[S{\xi}_n, {\xi}_n]\Big| \leq 4\Big|[S{\zeta}_n, P{\zeta}_n]\Big| + 2\|S\|\|{\eta}_n\|.$$
Since $\lim_{n\rightarrow\infty}[P{\zeta}_n, S{\zeta}_n] = 0$ and $\lim_{n\rightarrow\infty} \|{\eta}_n\| = 0$, from the above equality we get
$\lim_{n\rightarrow\infty}[S{\xi}_n, {\xi}_n] = 0$.

The converse is trivial.
\end{proof}
%%%%%%%%%%%%%%%%%%%%%%%%%%%%%%%%%%%%%%%%%%%%%%%%%%%%%
\begin{theorem}\label{th.026}
Let $T, S\in\B(\mathscr{H})$. Then the following statements are equivalent\textup{:}
\begin{itemize}
  \item[(i)] $T\perpb S$.
  \item[(ii)] $\|T + \lambda S\|^2 \geq \|T\|^2 + |\lambda|^2\,m^2(S) \qquad (\lambda\in\Bbb C),$
\end{itemize}
where $m(S)$ is the minimum modulus of $S$.
\end{theorem}
\begin{proof}
(i)$\Rightarrow$(ii) Let $T\perpb S$ and $\dim \mathscr{H} = \infty$. By (\ref{id.1.602}), there exists a sequence of unit vectors $({\xi}_n)\subset \mathscr{H}$ such that $\lim_{n\rightarrow\infty} \|T{\xi}_n\| = \|T\|$ and $\lim_{n\rightarrow\infty}[T{\xi}_n, S{\xi}_n] = 0.$
We have
\begin{align*}
\|T + \lambda S\|^2 &\geq \|(T + \lambda S){\xi}_n\|^2
\\& = \|T{\xi}_n\|^2 + \overline{\lambda}[T{\xi}_n, S{\xi}_n] + \lambda[S{\xi}_n, T{\xi}_n] + |\lambda|^2\|S{\xi}_n\|^2,
\end{align*}
for all $\lambda\in\Bbb C$ and $n\in\Bbb N$. Thus
$$\|T + \lambda S\|^2 \geq \|T\|^2 + |\lambda|^2\lim_{n\rightarrow\infty}\sup \|S{\xi}_n\|^2 \geq \|T\|^2 + |\lambda|^2\,m^2(S) \qquad (\lambda\in\Bbb C).$$
When $\dim \mathscr{H} < \infty$, by using (\ref{id.1.603}), we can similarly prove the statement (ii).

(ii)$\Rightarrow$(i) This implication is trivial.
\end{proof}
%%%%%%%%%%%%%%%%%%%%%%%%%%%%%%%%%%%%%%%%%%%%%%%%
\begin{remark}
Notice that for $S\in\B(\mathscr{H})$ it is straightforward to show that $m(S) > 0$ if and only if $S$ is bounded below, or equivalently, $S$
is left invertible. So in the implication (i)$\Rightarrow$(ii) of Theorem \ref{th.026}, if $S$
is left invertible then $m(S) > 0$.
\end{remark}
%%%%%%%%%%%%%%%%%%%%%%%%%%%%%%%%%%%%%%%%%%%%%%%%%%%%%%
It is well known that Pythagoras' equality does not hold in $\B(\mathscr{H})$.
The following result is a kind of Pythagorean inequality for bounded linear operators.
\begin{corollary}\label{cr.030}
Let $T, S\in\B(\mathscr{H})$ with $m(S)> 0$. Then there exists a unique $\gamma\in\Bbb C$, such that
$$\Big\|(T + \gamma S) + \lambda S \Big\|^2 \geq \Big\|T + \gamma S \Big\|^2 + |\lambda|^2m^2(S) \qquad (\lambda\in\Bbb C).$$
\end{corollary}
\begin{proof}
The function $\lambda \longmapsto \|T + \lambda S\|$ attains its minimum at, say, $\gamma$ (there may be of course
many such points) and hence $T + \gamma S\perpb S$.
So, by Theorem \ref{th.026}, we have
$$\Big\|(T + \gamma S) + \lambda S \Big\|^2 \geq \Big\|T + \gamma S \Big\|^2 + |\lambda|^2\,m^2(S) \qquad (\lambda\in\Bbb C).$$
Now, suppose that $\xi$ is another point satisfying the inequality
$$\Big\|(T + \xi S) + \lambda S \Big\|^2 \geq \Big\|T + \xi S \Big\|^2 + |\lambda|^2\,m^2(S) \qquad (\lambda\in\Bbb C).$$
Choose $\lambda = \gamma - \xi$ to get
\begin{align*}
\Big\|T + \gamma S \Big\|^2 &= \Big\|(T + \xi S) + (\gamma - \xi) S \Big\|^2
\\& \geq \Big\|T + \xi S \Big\|^2 + |\gamma - \xi|^2\,m^2(S)
\\& \geq \Big\|T + \gamma S \Big\|^2 + |\gamma - \xi|^2\,m^2(S).
\end{align*}
Hence $0 \geq |\gamma - \xi|^2\,m^2(S)$. Since $m^2(S)> 0$, we get $|\gamma - \xi|^2 = 0$, or equivalently, $\gamma = \xi$. This shows that $\gamma$ is unique.
\end{proof}
%%%%%%%%%%%%%%%%%%%%%%%%%%%%%%%%%%%%%%%%%%%%%%%%%%%%%%%%%%%
Let $T\in\B(\mathscr{H})$. For every $S\in\B(\mathscr{H})$, it is easy to see that if there exists a unit vector $\xi\in \mathscr{H}$ such that $\|T\|\xi = |T|\xi$ and $S^*T\xi = 0$ then $T\perpm S$. The question is under which conditions the converse is true. When the Hilbert space is finite dimensional, it follows from Corollary \ref{rem.01} (iii) that there exists a unit vector $\xi\in \mathscr{H}$ such that $\|T\|\xi = |T|\xi$ and $S^*T\xi = 0$.

The following example shows that the finite dimensionality in the statement (iii) of Corollary \ref{rem.01} is essential.
%%%%%%%%%%%%%%%%%%%%%%%%%%%%%%%%%%%%%%%%%%%%%%%%
\begin{example}\label{ex.0224}
Consider operators $T, S\,:\ell^2\longrightarrow \ell^2$ defined by
$$T(\xi_1, \xi_2, \xi_3, \cdots) = (\frac{1}{2}\xi_1, \frac{2}{3}\xi_2, \frac{3}{4}\xi_3, \cdots)$$
and
$$S(\xi_1, \xi_2, \xi_3, \cdots) = (\xi_1, 0, 0, \cdots).$$
One can easily observe that $T \perpb S$ and $T^*S(\xi_1, \xi_2, \xi_3, \cdots) = \frac{1}{2}{\xi_1}^2\geq  0$. So, by (\ref{id.1.102}),
we get $T \perpm S$. But there does not exist $\xi\in\ell^2$ such that $\|T\|\xi = |T|\xi$.
\end{example}
We now settle the problem for any infinite dimensional Hilbert space.
The proof of Theorem \ref{th.0225} is a modification of one given by Paul et al. \cite[Theorem 3.1]{P.S.G}.
%%%%%%%%%%%%%%%%%%%%%%%%%%%%%%%%%%%%%%%%%%%%%%%%%%%%%%%%%
\begin{theorem}\label{th.0225}
Let $\dim \mathscr{H} = \infty$ and $T\in\B(\mathscr{H})$. If $\mathbb{S}_{{\mathscr{H}}_0} = \mathbb{M}_T$, where ${\mathscr{H}}_0$ is a finite dimensional
subspace of $\mathscr{H}$ and $\|T\|_{{{\mathscr{H}}_0}^\perp} = \sup\{\|T\xi\| :\, \xi\in{{\mathscr{H}}_0}^\perp,\, \|\xi\| = 1\}< \|T\|$, then for every $S\in\B(\mathscr{H})$ the following statements are equivalent\textup{:}
\begin{itemize}
\item[(i)] $T \perpm S$.
\item[(ii)] There exists a unit vector $\xi\in {\mathscr{H}}_0$ such that $\|T\xi\| = \|T\|$ and $S^*T\xi = 0$.
\item[(iii)] There exists a unit vector $\xi\in {\mathscr{H}}_0$ such that $\|T\|\xi = |T|\xi$ and $S^*T\xi = 0$.
\end{itemize}
\end{theorem}
\begin{proof}
(i)$\Rightarrow$(ii) Suppose (i) holds. By (\ref{id.1.602}), there exists a sequence of unit vectors
$\{\zeta_n\}$ in $\mathscr{H}$ such that
\begin{align}\label{id.17.981}
\lim_{n\rightarrow\infty} \|T{\zeta}_n\| = \|T\| \quad \mbox{and}\quad \lim_{n\rightarrow\infty} S^*T{\zeta}_n = 0.
\end{align}
For each $n\in\mathbb{N}$ we have
$$\zeta_n = \xi_n + \eta_n,$$
where $\xi_n \in \mathscr{H}_0$ and $\eta_n \in {\mathscr{H}_0}^\perp$.

Since $\mathscr{H}_0$ is a finite dimensional subspace and $\|\xi_n\| \leq 1$, so $\{\xi_n\}$ has a convergent subsequence
converging to some element of $\mathscr{H}_0$. Without loss of generality we assume that $\lim_{n\rightarrow\infty} \xi_n = \xi$. Since $\mathbb{S}_{{\mathscr{H}}_0} = \mathbb{M}_T$, so
\begin{align}\label{id.17.982}
\lim_{n\rightarrow\infty} \|T\xi_n\| = \|T\xi\| = \|T\|\,\|\xi\|
\end{align}
and
\begin{align}\label{id.17.983}
\lim_{n\rightarrow\infty} \|\eta_n\|^2 = \lim_{n\rightarrow\infty}( \|\zeta_n\|^2 - \|\xi_n\|^2 ) = 1 - \|\xi\|^2.
\end{align}
Now for each non-zero element $\xi_n \in \mathscr{H}_0$, by hypothesis $\frac{\xi_n}{\|\xi_n\|}\in \mathbb{S}_{{\mathscr{H}}_0} = \mathbb{M}_T$ and so $\|T\xi_n\| = \|T\|\,\|\xi_n\|$. Thus
\begin{align*}
\|T^*T\|\,\|\xi_n\|^2 = \|T\|^2\,\|\xi_n\|^2 = \|T\xi_n\|^2 = [T^*T\xi_n, \xi_n] \leq \|T^*T\xi_n\|\,\|\xi_n\| \leq \|T^*T\|.
\end{align*}
Hence $[T^*T\xi_n, \xi_n] = \|T^*T\xi_n\|\,\|\xi_n\|$. By the equality case of Cauchy--Schwarz inequality $T^*T\xi_n = \lambda_n\xi_n$ for some $\lambda_n\in\mathbb{C}$ and therefore
\begin{align}\label{id.17.985}
[T^*T\xi_n, \eta_n] = [T^*T\eta_n, \xi_n] = 0.
\end{align}
By (\ref{id.17.981}), (\ref{id.17.982}) and (\ref{id.17.985}) we have
\begin{align*}
\|T\|^2 &= \lim_{n\rightarrow\infty}\|T\zeta_n\|^2
\\& = \lim_{n\rightarrow\infty} [T^*T\zeta_n, \zeta_n]
\\& = \lim_{n\rightarrow\infty} \Big([T^*T\xi_n, \xi_n] + [T^*T\xi_n, \eta_n] + [T^*T\eta_n, \xi_n] + [T^*T\eta_n, \eta_n]\Big)
\\& = \lim_{n\rightarrow\infty} \|T\xi_n\|^2 + \lim_{n\rightarrow\infty} \|T\eta_n\|^2 = \|T\|^2\,\|\xi\|^2 + \lim_{n\rightarrow\infty} \|T\eta_n\|^2,
\end{align*}
whence by (\ref{id.17.983}) we reach
\begin{align}\label{id.17.984}
\lim_{n\rightarrow\infty} \|T\eta_n\|^2 = \|T\|^2(1 - \|\xi\|^2) = \|T\|^2\,\lim_{n\rightarrow\infty} \|\eta_n\|^2.
\end{align}
By the hypothesis $\|T\|_{{\mathscr{H}_0}^\perp} < \|T\|$ and so by (\ref{id.17.984}) there does
not exist any non-zero subsequence of $\{\|\eta_n\|\}$. So we conclude that $\eta_n = 0$ for all $n\in\mathbb{N}$. Hence (\ref{id.17.981}), (\ref{id.17.983}) imply
$$\|\xi\| = 1, \qquad \|T\xi\| = \|T\|\quad \mbox{and}\quad S^*T\xi = 0.$$
(ii)$\Rightarrow$(iii) This implication follows from the proof of Corollary \ref{rem.01}.\\
(iii)$\Rightarrow$(i) This implication is trivial.
\end{proof}
%%%%%%%%%%%%%%%%%%%%%%%%%%%%%%%%%%%%%%%%%%%%%%%%%%%%%%%%
\begin{corollary}\label{cr.0231}
Let $\dim \mathscr{H} = \infty$ and $T\in\B(\mathscr{H})$. If $\mathbb{S}_{{\mathscr{H}}_0} = \mathbb{M}_T$, where ${\mathscr{H}}_0$ is a finite dimensional
subspace of $\mathscr{H}$ and $\|T\|_{{{\mathscr{H}}_0}^\perp} < \|T\|$, then there exists a unit vector $\xi\in {\mathscr{H}}_0$ such that $\|T\|\xi = |T|\xi$ and $\|T\|^2\,T^*T\xi = (T^*T)^2\xi$.
\end{corollary}
\begin{proof}
By (\ref{id.1.104}), $T\perpm \big(\|T\|^2T - TT^*T\big)$. So, by Theorem \ref{th.0225},
there exists a unit vector $\xi\in {\mathscr{H}}_0$ such that $\|T\|\xi = |T|\xi$ and $\big(\|T\|^2T - TT^*T\big)^*T\xi = 0$. Thus $\|T\|^2\,T^*T\xi = (T^*T)^2\xi$.
\end{proof}
%%%%%%%%%%%%%%%%%%%%%%%%%%%%%%%%%%%%%%%%%%%%%%%%%%%%%%%%%%
\begin{corollary}\label{cr.0226}
Let $\dim \mathscr{H} = \infty$ and let $T\in\B(\mathscr{H})$ be a nonzero positive operator. If $\mathbb{S}_{{\mathscr{H}}_0} = \mathbb{M}_T$, where ${\mathscr{H}}_0$ is a finite dimensional
subspace of $\mathscr{H}$ and $\|T\|_{{{\mathscr{H}}_0}^\perp} < \|T\|$, then for every $S\in\B(\mathscr{H})$
the following statements are equivalent\textup{:}
\begin{itemize}
\item[(i)] $T \perpm S$.
\item[(ii)] There exists a unit vector $\xi\in {\mathscr{H}}_0$ such that $T\xi = \|T\|\xi$ and $S^*\xi = 0$.
\end{itemize}
\end{corollary}
\begin{proof}
Obviously, (ii)$\Rightarrow$(i).

Suppose (i) holds. By Theorem \ref{th.0225}, there exists a unit vector $\xi\in {\mathscr{H}}_0$ such that $\|T\xi\| = \|T\|$ and $S^*T\xi = 0$.
Since $T\geq0$, $\|T\xi\| = \|T\| \Leftrightarrow T\xi = \|T\|\xi$. Therefore, $S^*T\xi = 0 \Leftrightarrow S^*\xi = 0$, as $T\neq 0$.
\end{proof}
%%%%%%%%%%%%%%%%%%%%%%%%%%%%%%%%%%%%%%%%%%%%%%%%%%%%%%%%%%%%%%%%%%%%%%%%%%%%%%%%%%
\section{An approximate strong Birkhoff–-James orthogonality}
Recall that in an inner product $\A$-module $V$ and for $\varepsilon\in[0, 1)$, we say $x, y$ are
approximate strongly Birkhoff–-James orthogonal, in short $x\perpp y$, if
$$\|x + ya\|^2\geq\|x\|^2-2\varepsilon\,\|a\|\,\|x\|\,\|y\|\qquad (a\in\A).$$
The following proposition states some basic properties of the relation $\perpp$.
%%%%%%%%%%%%%%%%%%%%%%%%%%%%%%%%%%%%%%%%%%%%%%%%%%%%
\begin{proposition}\label{pr.042}
Let $\varepsilon\in[0, \frac{1}{2})$ and $V$ be an inner product $\A$-module. Then the following statements hold for every $x, y \in V$\textup{:}
\begin{itemize}
\item[(i)] $x\perpp x \Leftrightarrow x=0$.
\item[(ii)] $x\perpp y \Rightarrow \alpha x\perpp \beta y$ for all $\alpha, \beta\in\mathbb{C}$.
\item[(iii)] $x\perpr y \Rightarrow x\perpp y$.
\item[(iv)] $x\perpp y \Rightarrow x\perps y$.
\item[(v)] $x\perpp y \Leftrightarrow x\perps ya$ for all $a\in \A$.
\end{itemize}
\end{proposition}
\begin{proof}
(i) Let $x\perpp x$. Also, suppose that $(e_i)_{i\in I}$ is an approximate unit for $\A$. We have
$$\|x -xe_i\|^2\geq\|x\|^2-2\varepsilon\,\|-e_i\|\,\|x\|\,\|x\| \qquad (i\in I).$$
Since $\lim\limits_{i}\|x -xe_i\| = 0$ and $\|e_i\|=1$, we get $(1-2\varepsilon)\|x\|^2\leq 0$. Thus $x=0$.

The converse is obvious.\\
(ii) Let $x\perpp y$ and let $\alpha, \beta\in\mathbb{C}$. Excluding the obvious case $\alpha=0$ we have
\begin{align*}
\|\alpha x + \beta ya\|^2&= |\alpha|^2\left\|x + y\frac{\beta}{\alpha}a\right\|^2
\\&\geq |\alpha|^2\left(\|x\|^2-2\varepsilon\,\|a\|\,\|x\|\,\left\|\frac{\beta}{\alpha} y\right\| \right)
\\&= \|\alpha x\|^2-2\varepsilon \,\|a\|\,\|\alpha x\|\,\|\beta y\|.
\end{align*}
Hence $\alpha x\perpp \beta y$.\\
(iii) Let $x\perp^\varepsilon y$. For any $a\in \A$ we have
\begin{align*}
\|x + ya\|^2 &= \|\langle x+ya, x+ya\rangle\|
\\& = \|\langle x, x\rangle+\langle ya, ya\rangle+\langle x, ya\rangle+\langle ya, x\rangle \|
\\&\geq \|\langle x, x\rangle+\langle ya, ya\rangle \|-\|\langle x, ya\rangle+\langle ya, x\rangle \|
\\&\geq \|\langle x, x\rangle \|-\|\langle x, ya\rangle+\langle ya, x\rangle \|
\\&\geq \|x\|^2- \|\langle x, ya\rangle \| - \|\langle ya, x\rangle \|
\\&\geq \|x\|^2- 2\,\|a\|\,\|\langle x, y\rangle \|
\\&\geq \|x\|^2- 2\varepsilon \,\|a\|\,\|x\|\,\|y\|.
\end{align*}
Thus $\|x + ya\|^2\geq \|x\|^2- 2\varepsilon\,\|a\|\, \|x\|\,\|y\|$, or equivalently, $x\perpp y$.\\
(iv) Let $x\perpp y$. Hence for any $\lambda\in\mathbb{C}$ and an approximate unit $(e_i)_{i\in I}$ for $\A$ we have
\begin{align*}
\left(\|x + \lambda y\| + |\lambda|\|ye_i-y\|\right)^2&\geq\|x + \lambda ye_i\|^2
\\& \geq\|x\|^2- 2\varepsilon \,\|\lambda e_i\|\,\|x\|\,\|y\|
\\& \geq \|x\|^2- 2\varepsilon\,|\lambda|\,\|x\|\,\|y\|.
\end{align*}
Since $\lim\limits_{i}\|ye_i-y\| = 0$, whence we get $\|x + \lambda y\|^2\geq \|x\|^2- 2\varepsilon\,|\lambda|\,\|x\|\,\|y\|$, or equivalently, $x\perp^\varepsilon y$.\\
(v) Let $x\perpp y$ and let $(e_i)_{i\in I}$ be an approximate unit for $\A$. We have
\begin{align*}
\left(\|x + \lambda ya\| + \|\lambda yae_i-\lambda ya\|\right)^2& \geq \|x + \lambda yae_i\|^2
\\& \geq \|x\|^2-2\varepsilon\,\|\lambda ae_i\|\,\|x\|\,\|y\|
\\& \geq \|x\|^2-2\varepsilon\,|\lambda|\,\|a\|\,\|x\|\,\|y\|
\end{align*}
for all $a\in \A$ and all $\lambda\in\mathbb{C}$. Since $\lim\limits_{i}\|yae_i-ya\| = 0$, we obtain from the above inequality
$$\|x + \lambda ya\|^2 \geq \|x\|^2-2\varepsilon\,|\lambda|\,\|a\|\,\|x\|\,\|y\|,$$
for all $a\in \A$ and all $\lambda\in\mathbb{C}$. Thus $x\perps ya$ for all $a\in \A$.

The converse is trivial.
\end{proof}
Proposition \ref{pr.042} shows that in an arbitrary inner product $C^{*}$-module the relation $\perpr $
is weaker than the relation $\perpp $ and this relation is weaker than the relation $\perps$, but the converses are not true in general (see the example below).
%%%%%%%%%%%%%%%%%%%%%%%%%%%%%%%%%%%%%%%%%%%%%%%%%%%%%%%%
\begin{example}\label{ex.043}
Suppose that $\varepsilon\in[0 , \frac{1}{2})$. Consider $\mathbb{M}_2(\mathbb{C})$, regarded as an inner product $\mathbb{M}_2(\mathbb{C})$-module. Let
$I = \begin{bmatrix}
1 & 0 \\
0 & 1
\end{bmatrix}, A = \begin{bmatrix}
-1 & 0 \\
0 & 1
\end{bmatrix}$ and
$B = \begin{bmatrix}
1 & 0 \\
0 & 0
\end{bmatrix}$. Then
\begin{align*}
\|I + \lambda A\|^2 &= \left\|\begin{bmatrix}
1-\lambda & 0 \\
0 & 1+\lambda
\end{bmatrix}\right\|^2
\\& = \left(\max\{|1-\lambda|, |1+\lambda|\}\right)^2
\\&\geq1\geq 1-2\varepsilon\,|\lambda| = \|I\|^2- 2\varepsilon \,|\lambda|\,\|I\|\,\|A\|
\end{align*}
for all $\lambda\in\mathbb{C}$. Hence $I\perps A$. But not $I\perpp A$ since
$$\|I + A(-A)\|^2 = 0 < 1-2\varepsilon =  \|I\|^2- 2\varepsilon\,\|-A\| \|I\|\,\|A\|.$$
On the other hand, for any $C=\begin{bmatrix}
c_1 & c_2 \\
c_3 & c_4
\end{bmatrix}$ we have
\begin{align*}
\|I + BC\|^2 &= \left\|\begin{bmatrix}
1+c_1 & c_2 \\
0 & 1
\end{bmatrix}\right\|
\\& = \left[\frac{1}{2}\left(|1+c_1|^2+|c_2|^2+1\right)+ \frac{1}{2}\sqrt{\left(|1+c_1|^2+|c_2|^2+1\right)^2-4|1+c_1|^2}\right]^{\frac{1}{2}}
\\&\geq1\geq 1-2\varepsilon\,\|C\|\,\|B\| = \|I\|^2- 2\varepsilon \,\|C\|\,\|I\|\,\|B\|.
\end{align*}
Therefore $I\perpp B$. But not $I\perpr B$ since
$$\|\langle I, B\rangle\|= \|B\|=1 > \varepsilon = \varepsilon \|I\|\|B\|.$$
\end{example}
By combining Proposition \ref{pr.042} (iv) and \cite[Theorem 3.5]{M.T} we obtain the following result (see also \cite{Ch.3, F.M.P, WON}).
%%%%%%%%%%%%%%%%%%%%%%%%%%%%%%%%%%%%%%%%%%%%%%%%%%
\begin{corollary}\label{cr.044}
Let $V, W$ be inner product $\A$-modules, $\varepsilon\in[0 , \frac{1}{2})$ and $T\,: V \longrightarrow W$
a linear mapping satisfying $x \perpb y \Longrightarrow Tx\perpp Ty$. Then
$$(1 - 16\varepsilon)\,\|T\|\,\|x\| \leq \|Tx\| \leq \|T\|\,\|x\| \qquad (x\in V).$$
\end{corollary}
%%%%%%%%%%%%%%%%%%%%%%%%%%%%%%%%%%%%%%%%%%%%%
\begin{proposition}\label{pr.045}
Let $\varepsilon\in[0 , 1)$. Let $x, y$ be elements in an inner product $\A$-module $V$ such that $\langle x, x\rangle\perpp \langle x, y\rangle$, then $x\perpp y$.
\end{proposition}
\begin{proof}
We assume that $x\neq 0$. Since $\langle x, x\rangle\perpp \langle x, y\rangle$ therefore for every $a\in\A$ we have
$$\|\langle x, x\rangle + \langle x, y\rangle a\|^2\geq \|\langle x, x\rangle\|^2-2\varepsilon\,\|a\|\,\|\langle x, x\rangle\|\|\langle x, y\rangle\|$$
or equivalently,
$$\|\langle x, x + ya\rangle\|^2\geq \|x\|^4-2\varepsilon\,\|a\|\, \|x\|^2\,\|\langle x, y\rangle \|.$$
Hence we get $$\|x\|^2\,\|x + ya\|^2\geq \|x\|^4-2\varepsilon \,\|a\|\,\|x\|^3\,\|y\| \qquad (a\in\A).$$
Since $\|x\|^2\neq0$ we obtain from the above inequality
$$\|x + ya\|^2\geq \|x\|^2-2\varepsilon\,\|a\|\,\|x\|\,\|y\|\qquad (a\in\A).$$
Thus $x\perpp y$.
\end{proof}
%%%%%%%%%%%%%%%%%%%%%%%%%%%%%%%%%%%%%%%%%%%%%%
\begin{proposition}\label{cr.046}
Let $x, y$ be two elements in an inner product $\A$-module $V$ and let $\varepsilon\in[0 , 1)$. If there exists a state $\varphi$ on $\A$ such that $\varphi(\langle x, x\rangle) = \|x\|^2$ and $|\varphi(\langle x, y\rangle a)|\leq \varepsilon \,\|a\|\,\|x\|\,\|y\|$ for all $a\in\A$, then $x\perpp y$.
\end{proposition}
\begin{proof}
We assume that $x\neq 0$. Let $a\in\A$. By the Cauchy--Schwarz inequality, we have
\begin{align*}
\|x\|^2 &= \varphi(\langle x, x\rangle)
\\& = |\varphi(\langle x, x + ya\rangle) - \varphi(\langle x, ya\rangle)|
\\& \leq |\varphi(\langle x, x + ya\rangle)| + |\varphi(\langle x, ya\rangle)|
\\& \leq \sqrt{\varphi(\langle x, x\rangle)\varphi(\langle x + ya, x + ya\rangle)}+ \varepsilon\,\|a\|\, \|x\|\,\|y\|
\\& \leq \|x\|\,\|x + ya\| + \varepsilon\,\|a\|\,\|x\|\,\|y\|.
\end{align*}
Thus $\|x\|^2 \leq \|x\|\,\|x + ya\| + \varepsilon\,\|a\|\,\|x\|\,\|y\|$, i.e, $\|x + ya\|\geq \|x\| - \varepsilon\,\|a\|\,\|y\|$. We consider two cases:\\
If $\|x\| - \varepsilon\,\|a\|\,\|y\| \geq 0$, then we get
\begin{align*}
\|x + ya\|^2 &\geq \left(\|x\| - \varepsilon\,\|a\|\,\|y\|\right)^2
\\& = \|x\|^2 - 2\varepsilon \,\|a\|\,\|x\|\,\|y\| + {\varepsilon}^2\,\|a\|^2\,\|y\|^2
\\& \geq\|x\|^2 - 2\varepsilon \,\|a\|\,\|x\|\,\|y\|.
\end{align*}
If $\|x\| - \varepsilon\,\|a\|\,\|y\| < 0$, then we reach
\begin{align*}
\|x + ya\|^2 &\geq 0 > \|x\|\left(\|x\| - \varepsilon\,\|a\|\,\|y\|\right)
\\& \geq \|x\|\left(\|x\| - \varepsilon\,\|a\|\,\|y\|\right) - \varepsilon\,\|a\|\,\|x\|\,\|y\|
\\& = \|x\|^2 - 2\varepsilon\,\|a\|\,\|x\|\,\|y\|.
\end{align*}
Hence $x\perpp y$.
\end{proof}
%%%%%%%%%%%%%%%%%%%%%%%%%%%%%%%%%%%%%%%%%%%%%
\begin{proposition}\label{pr.048}
Let $x, y$ be two elements in an inner product $\A$-module $V$ and let $\varepsilon\in[0 , \frac{1}{2})$. If $x\perpp y$ then there exists a state $\varphi$ on $\A$ such that
$$| \varphi(\langle x, y\rangle a)|\leq \sqrt{2\varepsilon}\,\|a\|\, \|x\|\,\|y\| \qquad(a\in\A).$$
\end{proposition}
\begin{proof}
Suppose that $x\perpp y$. Because of the homogeneity of relation $\perpp$ we may assume, without loss of generality, that $\|x\|=\|y\|=1$. Then, for arbitrary $a\in\A$ we have
$$\|x + ya\|^2\geq 1- 2\varepsilon\,\|a\|\, \|y\|.$$
Since $\|-\langle y, x\rangle\|\leq \|y\|\,\|x\|=1$, hence for $a = - \langle y, x\rangle \in \A$ we get
$$\|x - y\langle y, x\rangle\|^2\geq 1 - 2\varepsilon .$$
On the other side, by Theorem 3.3.6 of \cite{mor}, there is $\varphi \in \mathcal{S}(\A)$ such that
$$\varphi\Big(\Big\langle x - y\langle y, x\rangle, x - y\langle y, x\rangle \Big\rangle \Big)= \|x - y\langle y, x \rangle\|^2.$$
Also, we have
\begin{align*}
\varphi&\Big(\Big\langle x - y\langle y, x\rangle, x - y\langle y, x\rangle \Big\rangle \Big)
\\& = \varphi(\langle x, x\rangle) - 2 \varphi(\langle x, y\rangle \langle y, x\rangle) + \varphi(\langle x, y\rangle \langle y, y\rangle \langle y, x\rangle)
\\& \leq \|x\|^2 - 2 \varphi(\langle x, y\rangle \langle y, x\rangle) + \varphi(\langle x, y\rangle \|y\|^2 \langle y, x\rangle)
\\& = 1- \varphi(\langle x, y\rangle \langle y, x\rangle),
\end{align*}
so, we get
\begin{align*}
1- \varphi(\langle x, y\rangle \langle y, x\rangle) \geq \varphi\Big(\Big\langle x - y\langle y, x\rangle, x - y\langle y, x\rangle \Big\rangle \Big)= \|x - y\langle y, x \rangle\|^2 \geq 1 - 2\varepsilon.
\end{align*}
Therefore $\varphi(\langle x, y\rangle \langle y, x\rangle)\leq 2\varepsilon$. Now, by the Cauchy--Schwarz inequality, we reach
$$|\varphi(\langle x, ya\rangle)|\leq \sqrt{\varphi(\langle x, y\rangle\langle y, x\rangle)\varphi(a^\ast a)}\leq \sqrt{2\varepsilon}\|a\| \qquad (a\in\A).$$
\end{proof}
%%%%%%%%%%%%%%%%%%%%%%%%%%%%%%%%%%%%%%%%%%%%%%%%%%%%%%%%%%%%%%%%

\bibliographystyle{amsplain}

\end{document}